\documentclass[12pt]{amsart}

\usepackage{amsthm, amsmath, amssymb, url, verbatim}
\usepackage[margin=1.2in]{geometry}
{\providecommand{\noopsort}[1]{} 

\frenchspacing

\theoremstyle{plain}
\newtheorem{theorem}{Theorem}[section]

\newtheorem{lemma}[theorem]{Lemma}

\newtheorem{conjecture}[theorem]{Conjecture}
\newtheorem{question}[theorem]{Question}

\newtheorem{main}{Theorem}
\theoremstyle{definition}
\newtheorem{definition}[theorem]{Definition}

\theoremstyle{remark}

\newtheorem{example}[theorem]{Example}
\numberwithin{equation}{section}

\newcommand{\Z}{\mathbb{Z}}
\newcommand{\R}{\mathbb{R}}\newcommand{\C}{\mathbb{C}}\newcommand{\HH}{\mathbb{H}}

\newcommand{\pp}{\mathbb{P}}
\newcommand{\s}{\mathbb{S}}

\DeclareMathOperator{\Sq}{Sq}
\DeclareMathOperator{\PP}{P}
\DeclareMathOperator{\im}{im}

\newcommand{\of}[1]{\left(#1\right)}

\title[Extending Adams' theorem]{Extending Adams' theorem from singly generated to periodic cohomology}

\author{John R. Harper}
\address{Department of Mathematics, University of Rochester, Rochester, NY 14627, USA}
\email{harperjohn973@gmail.com}

\author{Lee Kennard}
\address{Department of Mathematics, Syracuse University, Syracuse, NY 13244, USA}
\email{ltkennar@syr.edu}

\begin{document}

\begin{abstract}
In 1960, J.F. Adams introduced secondary cohomology operations that are defined on cohomology elements on which sufficiently many Steenrod algebra elements vanish. This led to his theorem on singly generated cohomology rings, which in turn led to his celebrated resolution of the Hopf invariant one problem. Here we advertise a conjecture that would extend Adams' result and prove it in a special case.
\end{abstract}

\maketitle

\section{Introduction}
\bigskip

J.F. Adams' theorem on singly generated cohomology rings was the key step in his resolution of the Hopf invariant one problem. The theorem is as follows (see \cite{Adams60}):

\begin{theorem}[Adams, 1960]\label{thm:Adams}
Let $X$ be a topological space. If $H^*(X;\Z_2)$ is isomorphic to the polynomial ring $\Z_2[x]$, then $x$ has degree $k \in \{1,2,4\}$.
\end{theorem}

Adams' theorem also holds for truncated polynomial algebras on one generator whose square is non-trivial. It allows for $k = 8$ in addition, so long as the third power of the generator vanishes. All values of $k$ permitted by Adams' theorem were long known to be realized. Examples include the classifying spaces $B\mathsf{O}(1) = \R\pp^\infty$, $B\mathsf{U}(1) = \C\pp^\infty$, and $B\mathsf{Sp}(1) = \HH\pp^\infty$, their truncated analogues, and the Cayley plane.

The conjectured extension has the same conclusion but a weaker assumption. For simplicity, we restrict our attention to an extension in the non-truncated case. We consider \textit{$k$--periodic} cohomology rings:

\begin{definition}\label{def:periodicity}
For a connected topological space $X$, we say that a non-zero element $x\in H^k(X;\Z_2)$ induces periodicity if the map 
	$H^i(X;\Z_2) \to H^{i+k}(X;\Z_2)$ 
induced by multiplication by $x$ is an isomorphism for all $i \geq 0$. When such an element exists, we say that $H^*(X;\Z_2)$ is $k$-periodic. \end{definition}

\begin{example}[Acyclic spaces] 
If we were to only require surjectivity in degree zero and allow $x = 0$ in the definition, we would be including mod 2 acyclic spaces (or mod 2 homology spheres in the manifold setting, properly formulated). For our purposes, these cases are trivial, so we require here that $x \neq 0$ and hence generates $H^k(X;\Z_2)$.
\end{example}

\begin{example}
In addition to singly generated cohomology rings, the spaces $\s^1 \times \C\pp^\infty$, $N^2 \times \HH\pp^\infty$, and $N^3 \times \HH\pp^\infty$ have $4$-periodic $\Z_2$-cohomology rings, where $N$ is any connected manifold of dimension 2 or 3 in the last two cases, respectively.
\end{example}

Observe that, if $H^*(X;\Z_2) \cong \Z_2[x]$, then $x$ is nonzero and induces periodicity. Moreover, $x$ has minimal degree among all such elements. The following conjecture therefore extends Adams' theorem (see \cite[Conjecture 6.3]{Kennard13}):

\begin{conjecture}[$\Z_2$ Periodicity Conjecture]\label{con:2}
Let $X$ be a topological space such that $H^*(X;\Z_2)$ is $k$-periodic for some $k \ge 1$. If $k$ is the minimum period, then $k \in \{1,2,4\}$.
\end{conjecture}

Evidence for this conjecture includes the proof that $k$ is a power of two (see \cite[Proposition 1.3]{Kennard13}). This result extended J. Adem's contribution to the Hopf invariant one problem to the case of periodic cohomology as in this conjecture (see \cite{Adem52}). Additional evidence exists for the odd prime analogue of the conjecture in the manifold case when $k = 2p$ (see Section \ref{sec:truncated}). In this article, we provide additional evidence:

\begin{main}\label{thm:2}
The periodicity conjecture holds for spaces $X$ satisfying the property that $H^{\frac k 2}(X; \Z_2) = 0$, where $k$ is the minimum period.
\end{main}

One recovers Adams' theorem from Theorem \ref{thm:2} by restricting to the case where $H^i(X; \Z_2) = 0$ for all $0 < i < k$. Our proof actually works under the more general assumption that the Steenrod square $\Sq^{\frac k 2}$ evaluates to zero on the group $H^k(X; \Z_2) \cong \Z_2$. This condition is implied by the one in Theorem \ref{thm:2} because the image lands in the group $H^{k + \frac k 2}(X; \Z_2)$, which by periodicity is isomorphic to $H^{\frac k 2}(X; \Z_2)$.

The proof of Theorem \ref{thm:2} follows Adams' strategy and, in particular, uses secondary cohomology operations. Unlike the singly generated case, however, even proving that these operations are defined requires proving certain vanishing results for the primary operations, i.e., Steenrod squares. The assumption in Theorem \ref{thm:2} gives us one such vanishing statement, and the main technical work in this article is to derive additional vanishing results for the action of the Steenrod algebra. Once this is done, we may apply Adams' basic strategy, together with results in \cite{Kennard13}, to finish the proof.

The odd prime analogue of Conjecture \ref{con:2} is the following:

\begin{conjecture}[$\Z_p$ Periodicity Conjecture]\label{con:p}
Let $X$ be a topological space and $p$ be an odd prime. If a non-zero element $x \in H^k(M; \Z_p)$ induces isomorphisms $H^i(X; \Z_p) \to H^{i+k}(X; \Z_p)$ by multiplication for all $i \geq 0$, and if $x$ has minimal degree among all such elements, then $k = 2 \lambda$ for some divisor $\lambda$ of $p - 1$.
\end{conjecture}

By \cite[Proposition 2.1]{Kennard13}, it is known that the minimum period is of the form $k = 2 \lambda p^a$ where $\lambda \mid p - 1$ and $a \geq 0$. Here we prove $a = 0$ under an additional vanishing assumption similar in analogy with Theorem \ref{thm:2}.

\begin{main}\label{thm:p}
In the setting of Conjecture \ref{con:p}, if $k = 2 \lambda p^a$ for some divisor $\lambda$ of $p - 1$ and some $a \geq 1$, then $H^{2 (p - 1) p^{a-1}}(X; \Z_p) \neq 0$.
\end{main}

In the special case where $H^i(X; \Z_p) = 0$ for all $0 < i < k$, where $k$ is the minimum period, one recovers the classical work of Liulevicius \cite{Liulevicius62} and Shimada-Yamanoshita \cite{ShimadaYamanoshita61} who proved the odd prime analogue of Adams' theorem on singly generated cohomology rings.

More generally, the proof of Theorem \ref{thm:p} shows that any counterexample $X$ to Conjecture \ref{con:p} satisfies $k = 2 \lambda p^a$ for some $\lambda \mid p - 1$ and $a \geq 1$ and $P^{(l-1)p^a + p^{a-1}} \neq 0$ on the group $H^k(X; \Z_p) \cong \Z_p$ for some $1 \leq l \leq \lambda$. This assumption is more general since the images of these Steenrod powers land in the group $H^{2 \lambda p^a + 2 (p - 1) (l - 1) p^a + 2 (p - 1) p^{a-1}}(X; \Z_p)$, which by periodicity is isomorphic to $H^{2 (p - 1) p^{a - 1}}(X; \Z_p)$.

The condition of periodic cohomology arises in Riemannian geometry. In a foundational paper, Wilking \cite{Wilking03} proved that totally geodesic inclusions of positively curved Riemannian manifolds are highly connected. In special situations, this yields $n$-connected inclusions of closed, orientable manifolds $N^n \to M^{n+k}$ where $k \leq n$. Combined with Poincar\'e duality, this implies the existence of an element $x \in H^k(M; \Z)$ such that multiplication by $x$ induces an isomorphism $H^i(M; \Z) \to H^{i+k}(M; \Z)$ for all $0 < i < n$. Moreover the map from $H^0(M; \Z)$ is surjective and the map into $H^{n+k}(M;\Z)$ is injective. Passing to coefficients in $\Z_2$, one finds that $M$ is a $\Z_2$-homology sphere or has $k$-periodic cohomology in a sense similar to Definition \ref{def:periodicity}.

In \cite[Theorem C]{Kennard13}, the second author used Steenrod's cohomology operations to refine the $\Z_p$ periodicity statements for all primes $p$ and proved, assuming further that $2k \leq n$, that $M$ is a rational homology sphere or has four-periodic rational cohomology. Together with vanishing results for the first and third Betti numbers in this setting, one obtains vanishing results for all odd Betti numbers. These are key steps in the verification by the second author, Wiemeler, and Wilking of Hopf's Euler Characteristic Positivity Conjecture for Riemannian metrics with isometry group of rank at least five (see \cite{KWW1}). The conjecture goes back to the 1930s, and this was the first evidence involving symmetry groups whose rank is not required to grow to infinity in the manifold dimension.

In \cite{Nienhaus-pre}, Nienhaus improved \cite[Theorem C]{Kennard13} in the special case where the periodicity is induced by a $n$-connected inclusion $N^n \subseteq M^{n+k}$ of closed, oriented manifolds. His technique uses characteristic classes and in particular allows one to drop the assumption of simply connected. In the further special case where the normal bundle to $N$ admits a complex structure, which can often be arranged in geometric applications, Nienhaus also relaxes the codimension assumption from $2k \leq n$ in \cite[Theorem C]{Kennard13} to $k \leq n$. As an application, Nienhaus proved one can relax the torus rank assumption in \cite{KWW1} from five to four. The significance of this advance is highlighted by the fact that the proof technique in both papers fails for rank two. Therefore Nienhaus' proof is in a sense either optimal already or only off by one.

\subsection*{Acknowledgements} 
This work was supported by NSF Grant DMS-2402129, the Simons Foundation's TSM program, and NSF Grant DMS-1928930 while the second author was in residence at SLMath in Berkeley, California, in Fall 2024.

\bigskip
\section{Classical proof in the singly generated case}\label{sec:ClassicalProof}
\bigskip

Let $X$ be a space with $H^*(X;\Z_2) \cong \Z_2[x]/(x^{q+1})$ for some $q \geq 2$ or $q = \infty$. Adem showed that, if $k \neq 2^a$, then there is a universal decomposition
	\[\Sq^k = \sum_{0 < i < k} \Sq^i \circ a_{i}\]
for some Steenrod algebra elements $a_{i}$. Evaluating on $x$, we conclude
	\[x^2 = \sum_{0 < i < k} \Sq^i(w_i)\]
for some cohomology elements $w_i \in H^{2k - i}(X;\Z_2)$. All of these groups are zero for $0 < i < k$, so $x^2 = 0$, a contradiction to $q \geq 2$.

Now assume $k = 2^a \geq 16$. Note that $\Sq^{2^i}(x) = 0$ for all $0 < 2^i < k$. Adams showed that there is a secondary decomposition 
	\[\Sq^{2^a} = \sum_{0 < 2^i < 2^a} \Sq^{2^i} \circ \Phi_i\]
that holds on the subspace
	\[H^{2^a}(X;\Z_2) \cap \bigcap_{0 < 2^i < 2^a} \ker(\Sq^{2^i})\]
and lands in the quotient 
	\[H^{2^{a+1}}(X;\Z_2)/\bigoplus_{0 < 2^i < 2^a} \im(\Sq^{2^i}).\] 
Evaluating on $x$, we get a conditional relation of the form
	\[x^2 = \sum_{0 < 2^i < k} \Sq^{2^i}(w_i),\]
where $w_i \in H^{2^{a+1}-2^i}(X;\Z_2)$. Once again, since these groups are zero, the $w_i$ are zero, and we have a contradiction.

For the case $k = 8$, it suffices to derive a contradiction if $x^3 \neq 0$. For degree reasons, $\Sq^1(x) = \Sq^2(x) = \Sq^4(x) = 0$. Following Gon\c calves \cite{Goncalves77}, we find that there are secondary cohomology operations $\Phi'_{0,3}$ and $\Phi_i$ satsifying the following two properties:
	\begin{enumerate}
	\item $\Phi_{0,3}'(\Sq^8(u)) = \Sq^{16}(u) + \sum_{0 < 2^i < 8} \Sq^{2^i}(\Phi_i(u))$ for all $u \in \bigcap_{0 < 2^i < 8} \ker(\Sq^{2^i})$ that holds modulo $\bigoplus_{0 < 2^i \leq 8} \im(\Sq^{2^i})$. 
	\item $\Phi_{0,3}'(u^2) = u^3$ modulo $\bigoplus_{0 < 2^i \leq 8} \im(\Sq^{2^i})$ for $u \in H^8(X;\Z_2) \cap \ker(\Sq^1, \Sq^2,\Sq^4)$.
	\end{enumerate}
Taking $u = x$ in the first and second equations, we conclude that
	\[x^3 = \sum_{0 < 2^i \leq 8} \Sq^{2^i}(w_i)\]
for some $w_i \in H^{24 - 2^i}(X;\Z_2)$ (see also \cite[Corollary 1.3]{LinWilliams90} for a refined decomposition). Most terms on the right-hand side vanish as in the previous case, so we have $x^3 = \Sq^8(w)$ for some $w \in H^{16}(X;\Z_2)$. This group is generated by $x^2$, so we again have a contradiction since $\Sq^8(x^2) = \of{\Sq^4(x)}^2 = 0$.

\bigskip
\section{Results on the general case}\label{sec:TheoremA}
\bigskip

Throughout this section, let $X$ be a connected topological space with $k$-periodic $\Z_2$-cohomology. We abbreviate $H^i(X;\Z_2)$ by $H^i$. We make the following assumptions:
	\begin{enumerate}
	\item $x \in H^k$ is non-zero and induces periodicity. In particular,
		\begin{enumerate}
		\item $H^0 \cong H^k \cong H^{2k} \cong\ldots$, and these groups are generated by powers of $x$.
		\item $H^i$ is arbitrary for $0 < i < k$.
		\item All other groups are determined by the isomorphisms $H^i \to H^{i+k}$ given by multiplication by $x$. In particular, any $y \in H^i$ with $i \geq k$ factors as $xy'$ for some $y' \in H^{i-k}$.
		\end{enumerate}
	\item The degree $k$ is minimal.
	\end{enumerate}

Detailed proofs of the first three lemmas below can be found in \cite{Kennard13}.

\begin{lemma}
By minimality, neither $x$ nor $x^2$ factors as $yz$ with $0 < \deg(y) < k$.
\end{lemma}

\begin{proof}[Proof sketch]
The definition of periodicity implies that, if $x$ or $x^2$ equals $yz$, then $y$ too induces periodicity. (Multiplication by $x$ equals multiplication by $y$ and then by $z$.)
\end{proof}

\begin{lemma}\label{lem:2}
There is no $y$ with $\deg(y) < k$ such that $x = \Sq^i(y)$. Similarly, there is no $y$ with $\deg(y) \not\in\{k,2k\}$ such that $x^2 = \Sq^i(y)$.
\end{lemma}

\begin{proof}[Proof sketch]
If there is such a decomposition, choose one with maximal $i$. Square both sides, and apply the Cartan formula and periodicity multiple times to conclude that $x$ factors. By the previous lemma, this is a contradiction.

To prove the second claim, note that $x^2 \neq 0$ by periodicity. Hence $\deg(y) > k$ and $y^2 = x y'$ by periodicity with $0 < \deg(y') < k$. By using the Cartan formula and periodicity, we then factor out and cancel an $x$. This reduces the claim to the first part of the lemma.
\end{proof}

\begin{lemma}\label{lem:3}
The degree $k$ is a power of two.
\end{lemma}

\begin{proof}[Proof sketch]
If $k$ is not a power of two, then $\Sq^k$ decomposes by the Adem relations. Evaluate such a relation on $x$. Use the fact that $x^2$ generates $H^{2k}$ and Lemma \ref{lem:2} to derive a contradiction.
\end{proof}

Guided by Adams \cite{Adams60}, we have a two-step strategy to prove the main estimate $k = 2^a \leq 8$:
	\begin{itemize}
	\item (Step 1) Prove that $\Sq^i(x) = 0$ for all $0 < i < k$.
	\item (Step 2) Assuming $k \geq 16$, conclude the existence of a conditional relation $\Sq^k(x) = \sum_{0 < i < k} \Sq^i(w_i)$ coming from a decomposition of $\Sq^k(x) = x^2$ in terms of secondary cohomology operations, and then conclude a contradiction.
	\end{itemize}

As we show below, Step 2 can be carried out with no problem. The issue is to prove that $\Sq^i(x) = 0$ for all $0 < i < k$ in Step 1. As it turns out, we can prove this for all $i$ under the assumption that it holds for one particular $i$, namely $i = \frac k 2$. The main technical lemma that implies our theorem is the following:

\begin{lemma}\label{lem:4}
If $\Sq^{k/2}(x) = 0$, then $k \in \{1,2,4,8\}$. Moreover, $k = 8$ only if $x^3 = 0$.
\end{lemma}

\begin{proof}
Assume $k = 2^a \geq 16$ and $\Sq^{k/2}(x) = 0$. First, we induct over $i$ to prove that $\Sq^i(x) = 0$ for $\frac{k}{2} < i < k$, then we induct over $i$ again to prove that $\Sq^i(x) = 0$ for $0 < i < \frac{k}{2}$ (see Section \ref{sec:Calculations} for details.) Given these calculations, there is a conditional relation of the form
	\[x^2 = \Sq^k(x) = \sum_{0 < i < k} \Sq^i(y_i)\]
for some cohomology elements $y_i$ (see \cite{Harper02}). Since $x^2$ generates $H^{2k}$ by periodicity, we conclude $x^2 = \Sq^i(y_i)$ for some $i$, a contradiction to Lemma \ref{lem:2}.

Now assume $k = 8$ and that $x^3 \neq 0$. The assumption of the lemma states that $\Sq^4(x) = 0$, and it follows as in the previous case that $\Sq^i(x) = 0$ for $4 < i < 8$ and then for $0 < i < 4$. Hence there is a decomposition 
	\[x^3 = \sum_{0 < 2^i \leq 8} \Sq^{2^i}(w_i).\]
If some term is non-trivial with $0 < 2^i < 8$, then we have a contradiction to the minimality of $8 = \deg(x)$ by Lemma \ref{lem:2}. The $2^i = 8$ term is also zero since $H^{16}(X;\Z_2)$ is generated by $x^2$, which implies that $\Sq^8(x^2) = x^2 x_4^2$ where $\Sq^4(x) = xx_4$. But again $x_4^2 = 0$ by minimality since $x$ does not factor.
\end{proof}

To complete the proof of the conjecture, it would suffice to prove that $\Sq^{\frac k 2}(x) = 0$. On the other hand, perhaps the conjecture is false, and we take a moment to discuss what a counterexample might look like. Suppose $X$ is a space that has, for example, $16$-periodic but not $8$-periodic cohomology. Periodicity implies the existence of generators $1$, $x$, and $x^2$ in degrees $0$, $16$, and $32$, respectively. In addition, Lemma \ref{lem:4} implies that $\Sq^8(x) \neq 0$, since otherwise the period of $16$ is not minimal. Applying periodicity once more, we find that $\Sq^8(x) = x y_8 \neq 0$ for some element $y_8$ in degree $8$. Moreover, there must be additional non-zero groups $H^i(X; \Z_2)$ with $i \not\equiv 0 \bmod 8$, because otherwise one has $x y_8 \in \ker(\Sq^i)$ for $i \in \{1,2,4,8\}$ and hence a conditional relation of the form
	\[\Sq^{16}(x y_8) = \sum_{0 \leq i \leq 3} \Sq^{2^i}(x^2 z_{8-2^i}),\]
which implies that $x^2 y_8 = \Sq^8(x^2) = \Sq^4(x)^2 = 0$, a contradiction to periodicity. This shows that a counterexample to Conjecture \ref{con:2} with $k = 16$ would need to have non-trivial cohomology in some degrees $i \not\equiv 0 \bmod 8$. We further explore the structure of potential counterexamples in Section \ref{sec:PotentialCounterexamples}.

\bigskip
\section{Calculations}\label{sec:Calculations}\label{sec:Calculations} 
\bigskip

Fix $x \in H^k(M;\Z_2)$ as in Lemma \ref{lem:4}. We first show by induction that $\Sq^i(x) = 0$ for $\frac{k}{2} \leq i < k$. Note that $\Sq^{k/2}(x) = 0$ by assumption. Write $i = \frac{k}{2} + \delta$ where $0 < \delta < \frac{k}{2}$. Since $k$ is a power of two, there is an Adem relation of the form
	\[\Sq^\delta \Sq^{k/2} = \Sq^i + \sum_{0<j\leq\delta/2} c_j \Sq^{i-j} \Sq^j.\]
Evaluating on $x$, we see that $\Sq^i(x) = 0$ if $\Sq^{i-j}(\Sq^j(x)) = 0$ for all $0 < j \leq \frac{\delta}{2}$. To prove this, we use periodicity to write $\Sq^j(x) = xx_j$ for some $x_j \in H^j(M;\Z_2)$ and then the Cartan formula to obtain
	\[\Sq^{i-j}(\Sq^j(x)) = \sum_{0 \leq h \leq j} \Sq^{i-j-h}(x) \Sq^h(x_j).\]
Observe that $\frac{k}{2} \leq i - j - h < i$, so the assumption $\Sq^{\frac k 2}(x) = 0$ and the induction hypothesis imply that every term in this sum is zero. This concludes the proof.

Second, we show by induction that $\Sq^i(x) = 0$ for $0 < i < \frac{k}{2}$. We may assume that $\Sq^j(x) = 0$ for all $0 < j < i$. Since $2i < 2(k-i)$, there is an Adem relation of the form
	\[\Sq^{2i}\Sq^{k - i} = c_0 \Sq^{k+i} + \sum_{0< j < i} c_j \Sq^{k + i - j} \Sq^j + \Sq^k \Sq^i.\]
Evaluating on $x$, the left-hand side vanishes by the first step of the proof. The first term on the right-hand side vanishes because $k+i$ is larger than the degree of $x$, and the terms in the middle vanish for all $0 < j < i$ by the induction hypothesis. Therefore we have $\Sq^k(\Sq^i(x)) = 0$. On the other hand, we can write $\Sq^i(x) = xx_i$ by periodicity and apply the Cartan formula to obtain
	\[0 = \Sq^k(\Sq^i(x)) 
	    = x^2x_i + \sum_{0< j \leq i} \Sq^{k - j}(x) \Sq^j(x_i).\]
Since $\frac{k}{2} < k - j < k$, the first step of the proof implies that $0 = x^2 x_i$. By periodicity, we conclude $0 = xx_i = \Sq^i(x)$.

\bigskip
\section{Odd primes}\label{sec:TheoremB}
\bigskip

Fix an odd prime $p$, and let $X$ be a topological space with periodic $\Z_p$-cohomology. Specifically, we assume $k \geq 1$ is minimal such that there exists a non-zero $x \in H^k(X; \Z_p)$ such that the maps $H^i(X; \Z_p) \to H^{i+k}(X; \Z_p)$ induced by multiplication by $x$ are isomorphisms for all $i \geq 0$.

By the $\Z_p$ Periodicity Theorem \cite[Proposition 2.1]{Kennard13}, we have 
	\[k = 2 \lambda p^a\]
for some divisor $\lambda$ of $p - 1$ and some $a \geq 0$. The $\Z_p$ Periodicity Conjecture claims that $a = 0$. To prove Theorem \ref{thm:p}, we proceed by contradiction. Specifically, we assume that $a \geq 1$ and that
	\[H^{2(p-1)p^a}(X; \Z_p) = 0,\] 
and we derive a contradiction to the minimality of $k$. 

By $k$-periodicity and the fact that $\lambda \mid p - 1$, the vanishing of this cohomology group implies
	\[H^{k + 2(p-1)(l-1)p^a + 2(p-1)p^{a-1}}(X; \Z_p) = 0\]
for all $1 \leq l \leq \lambda$. In particular, 
	\begin{equation}\label{eqn:Pizero}
	\PP^{(l-1)p^a + p^{a-1}}(x) = 0.
	\end{equation}
for all $1 \leq l \leq \lambda$ since this element lies in a vanishing cohomology group. Similar to the proof strategy for Theorem \ref{thm:2}, we bootstrap this vanishing of $P^i(x)$ for certain degrees to derive additional vanishing in other degrees.

\medskip
\noindent
{\bf Claim 1}: $P^i(x) = 0$ for all $(l-1)p^a + p^{a-1} < i < l p^a$ and for all $1 \leq l \leq \lambda$.

\begin{proof}
Fix $i$ as in the claim, and write $i = (l - 1)p^a + p^{a-1} + \delta$ for some $1 \leq l \leq \lambda$ and some $0 < \delta < (p-1)p^{a-1}$. By induction, we may assume the result holds for all $i'$ of the form $i' = (l' - 1) p^a + p^{a-1} + \delta'$ with $1 \leq l' \leq \lambda$ and $0 < \delta' < \delta$.

A calculation shows that $\delta < p(i - \delta)$. Hence there exists an Adem relation
	\[P^\delta P^{i - \delta} = c_0 P^i + \sum_{0 < j \leq \frac \delta p} c_j P^{i - j} P^j\]
for some $c_j \in \Z_p$. Moreover we have by Lucas' theorem that
	\[c_0 = \pm \binom{(p-1)(i-\delta) - 1}{\delta} = \pm \prod_{j\geq 0} \binom{n_j}{\delta_j} \bmod p\]
where $\sum n_j p^j$ and $\sum \delta_j p^j$ are the $p$-adic expansions of $(p-1)(i-\delta) - 1$ and $\delta$, respectively. Notice the coefficients $n_j$ satisfy $n_{a-1} = p - 2$ and $n_{a-2} = \ldots = a_0 = p - 1$, and the coefficients $\delta_j$ satisfy $\delta_j = 0$ for $j \geq a$ and $\delta_{a - 1} \leq p - 2$ as a result of the bound $\delta < (p-1)p^{a-1}$. Therefore $\binom{n_j}{\delta_j} \not\equiv 0 \bmod p$ for all $j\geq 0$, and so $c_0 \neq 0$.

Evaluating the above Adem relation on $x$, we have zero on the left-hand side by Equation \eqref{eqn:Pizero}. On the right-hand side, the first term is a non-zero multiple of $P^i(x)$, so finishing the proof requires that we show $P^{i-j}(P^j(x)) = 0$ for all $0 < j \leq \delta/p$. To do this, write $P^j(x) = x x_{2(p-1)j}$ using periodicity and then
	\[P^{i-j}(P^j(x)) = \sum_{0 \leq h \leq i - j} P^{i - j - h}(x) P^h(x_{2(p-1)j})\]
by the Cartan formula. For $h > (p - 1)j$, we have $P^h(x_{2 (p - 1) j)}) = 0$. For $h < (p - 1)j$ or for $h = (p-1)j$ and $0 < j < \delta/p$, we have $(l-1)p^a + p^{a-1} < i - j - h < i$ and so our induction hypothesis implies $P^{i - j - h}(x) = 0$. Finally for $h = (p - 1)j$ and $j = \delta/p$, we have $P^{i - j - h}(x) = P^{i - \delta}(x) = 0$ by another application of Equation \eqref{eqn:Pizero}.
\end{proof}

\noindent
{\bf Claim 2}: $P^i(x) = 0$ for all $(l - 1)p^a < i < (l - 1)p^a + p^{a-1}$ and for all $1 \leq l \leq \lambda$.

\begin{proof}
Let $i = (l - 1)p^a + \epsilon$ for some $1 \leq l \leq \lambda$ and $0 < \epsilon < p^{a-1}$. By induction assume $P^{i'}(x) = 0$ for all $i'$ of the form $(l' - 1) p^a + \epsilon'$ with $1 \leq l' \leq \lambda$ and $0 < \epsilon' < \epsilon$.

Note that $p\epsilon < p\of{\lambda p^a - (p - 1)\epsilon}$, so there is an Adem relation of the form
	\[P^{p \epsilon} P^{\lambda p^a - (p-1)\epsilon} 
	= \sum_{0 \leq \epsilon' \leq \epsilon} c_{\epsilon'} P^{\lambda p^a + \epsilon - \epsilon'}(P^{\epsilon'}(x))\]
for some constants $c_{\epsilon'}$. The left-hand side vanishes by Claim 1 since 
	\[(\lambda - 1) p^a + p^{a - 1} < \lambda p^a - (p-1)\epsilon < \lambda p^a.\]
On the right-hand side, the $\epsilon' = 0$ term vanishes because $2\of{\lambda p^a + \epsilon} = k + 2\epsilon$ is larger than the degree of $x$, and the terms with $0 < \epsilon' < \epsilon$ vanish by our induction hypothesis. Therefore we have
	\[0 = P^{\lambda p^a}\of{P^\epsilon(x)}.\]
Writing $P^\epsilon(x) = x x_{2(p-1)\epsilon}$ by periodicity, we can apply the Cartan formula to conclude
	\[0 = x^p x_{2 (p-1) \epsilon} + \sum_{0 < h \leq \lambda p^a} P^{\lambda p^a - h}(x) P^h(x_{2(p-1)\epsilon}).\]
For the terms on the right-hand side with $0 < h < (p-1)p^{a-1}$, we have $P^{p^a - h}(x) = 0$ by Claim 1, and for terms with $h \geq (p-1) p^{a-1}$, we have $2h > 2 (p-1) \epsilon$ and hence $P^h(x_{2(p-1)\epsilon}) = 0$. Therefore $0 = x^p x_{2(p-1)\epsilon}$. By periodicity, we have $0 = x x_{2(p-1)\epsilon} = P^\epsilon(x)$. This finishes the proof that $P^i(x) = 0$ if $l = 1$.

Now we assume $1 < l \leq \lambda$. We have $\epsilon < p(l-1)p^a$, so we have an Adem relation
	\begin{equation}\label{eqn:Ademl}
	P^\epsilon P^{(l-1) p^a} = \sum_{0 \leq \epsilon' \leq \epsilon/p} c_{\epsilon'} P^{i-\epsilon'} P^{\epsilon'}.
	\end{equation}
Evaluating on $x$, we find that all terms on the right-hand side with $0 < \epsilon' \leq \epsilon/p$ vanish by our induction hypothesis. Moreover, the $\epsilon' = 0$ term has coefficient
	\[c_0 = \pm \binom{(p-1)(l-1)p^a - 1}{\epsilon},\]
which can be shown to be non-zero by an argument similar to the argument in Case 1. Therefore the proof of $P^i(x) = 0$ reduces to the claim that the left-hand side of this equation evaluates to zero. To prove this, we apply periodicity and the Cartan formula to see that
	\[P^\epsilon \of{P^{(l-1)p^a}(x)} = c P^\epsilon\of{x^{1 + l \frac{p-1}{\lambda}}} = c \sum P^{\epsilon_1}(x)\cdots P^{\epsilon_m}(x),\]
for some $c \in \Z_p$, where $m = 1 + l \frac{p-1}{\lambda}$ and the sum runs over $\epsilon_j \geq 0$ whose sum is $\epsilon$. By induction, $P^{\epsilon_j}(x) = 0$ unless $\epsilon_j \in \{0, \epsilon\}$. Similarly $P^{\epsilon}(x) = 0$ by the $l = 1$ case, so the right-hand side vanishes.
\end{proof}

\noindent
{\bf Claim 3}: $\beta(x) = 0$, where $\beta$ is the Bockstein homomorphism associated to the short exact sequence $0 \to \Z_2 \to \Z_4 \to \Z_2 \to 0$ of coefficient groups.

\begin{proof}
Because $k = 2 \lambda p^a$ with $a \geq 1$, we have $1 \leq p(k/2-1)$. Therefore there is an Adem relation 
	\[P^1 \beta P^{\frac k 2 -1} = - \beta P^{\frac k 2} + P^{\frac k 2} \beta.\]
Evaluating on $x$, the left-hand side vanishes by Claim 1, as does the first term on the right-hand side since
	\[\beta(P^{\frac k 2}(x)) = \beta(x^p) = p x^{p-1}\beta(x) = 0.\]
Therefore
	\[0 = P^{\frac k 2}(xx_1) = x^p x_1\]
by Cartan's formula and Claim 1, where we have written $\beta(x) = x x_1$ using periodicity. Therefore $0 = xx_1 = \beta(x)$, as claimed.
\end{proof}

Given Claims 1, 2, and 3, we derive a contradiction and therefore finish the proof of Theorem \ref{thm:p}. Taking $l = 1$ in Claims 1 and 2 as well as in Equation \eqref{eqn:Pizero}, we have $P^i(x) = 0$ for all $0 < i < p^a$. Since we also have $\beta(x) = 0$ by Claim 3, \cite[Theorem 6.2.1.b]{Harper02} implies that we have a conditional relation of the form
	\[P^{p^a}(x) = \beta(w) + \sum_{0 < i < p^a} P^i(w_i)\]
for some cohomology elements $w$ and $w_i$. 

First we claim the left-hand side is a non-zero multiple of $x^{1 + \frac{p-1}{\lambda}}$. Indeed, $P^{p^a}(x)$ lives in the one-dimensional group spanned by this power of $x$, so the claim follows if $P^{p^a}(x)$ is non-zero. To see this, we use that $P^{\lambda p^a}(x) = x^p \neq 0$ together with the Adem relation of the form
	\[P^{(\lambda - 1)p^a} P^{p^a} = c_0 P^{\lambda p^a} + \sum_{0<i \leq (\lambda - 1)p^{a-1}} c_i P^{\lambda p^a - i} P^i\]
where 
	\[c_0 = \pm \binom{(p-1)p^a - 1}{(\lambda - 1)p^a} \equiv \pm\binom{p-2}{\lambda -1} \not\equiv 0 \bmod p.\]
Since $P^i(x) = 0$ by Claim 1 for all $i < p^a$ and hence for all $i \leq (\lambda - 1)p^{a-1}$, the claim that $P^{p^a}(x)$ is a non-zero multiple of $x^{1 + \frac{p-1}{\lambda}}$ holds.

As for the right-hand side, we can furthermore use Adem relations to decompose the Steenrod powers $P^i$ using elements of the form $P^{p^j}$ with $0 \leq j \leq a - 1$. Since this equation holds in a one-dimensional group, we derive an equation of the form
	\[x^{1 + \frac{p-1}{\lambda}} = \beta(z) \mathrm{~or~} P^{p^j}(z)\]
for some cohomology element $z$. In the latter case, we note that $p^j$ satisfies the estimate
	\[2 (p - 1) p^j \leq 2 (p - 1) p^{a - 1} < 2 p^a \leq 2 \lambda p^a = k,\]
so the proof of Lemma 2.4 in \cite{Kennard13} implies that $x$ decomposes non-trivially as either a product or a Steenrod power of an element of smaller degree. As in the $\Z_2$ setting, this implies a contradiction to the minimality of $k$ (see \cite[Lemmas 1.2 and 2.3]{Kennard13}). Similarly, if $x^{1 + \frac{p-1}{\lambda}} = \beta(z)$ for some cohomology element $z$, then a similar argument leads to a contradiction to the minimality of $k$ (see \cite[Lemma 3.5]{Kennard17}).

\bigskip
\section{Truncated and manifold cases}\label{sec:truncated}
\bigskip

For an $n$-dimensional Poincar\'e duality manifold $M$, periodic $\Z_2$-cohomology is usually required to satisfy the condition that the minimum period $k$ satisfies $2k \leq n$. This condition ensures that $x^2 \neq 0$, which is essential to proving restrictions on the degree even in the singly generated case. Moreover, if $\dim M \equiv m \bmod k$, then we have automatically that $2k + m \leq n$.

Similarly, for odd primes $p$, periodic $\Z_p$-cohomology for an $n$-manifold is not expected to be restricted unless $pk \leq n$. Again letting $m$ be the congruence class of $n$ modulo $p$, Poincar\'e duality implies that periodicity goes from degree $0$ to $pk+m$. We therefore suggest the following as analogues in the truncated case of the periodicity conjecture.

\begin{conjecture}[Periodicity Conjecture, Truncated Case]
Let $X^n$ be a space and $p$ be a prime. Assume $H^*(X; \Z_p)$ is $k$-periodic up to degree $c$, in the sense that the last isomorphism lands in degree $c$. Assume that $k$ is the minimum period and that $0 \leq m < k$ is the maximum degree with $H^m(X; \Z_2) \neq 0$.
	\begin{enumerate}
	\item If $p = 2$ and $c \geq 2k + m$, then $k \in \{1, 2, 4\}$ or $(k,c) = (8, 16+m)$.
	\item If $p \geq 3$ and $c \geq pk + m$, then $k = 2 \lambda$ for some divisor $\lambda$ of $p - 1$.
	\end{enumerate}
\end{conjecture}

As an example, the product of the Cayley plane $\mathbb{OP}^2$ and any closed, orientable manifold $N^m$ with $m \leq 7$ has dimension $16+m$ and has $8$-periodic but not $4$-periodic $\Z_2$-cohomology. Similarly, it has $8$-periodic but not $4$-periodic $\Z_3$ cohomology, which is permitted by the conjecture since $pk + m > \dim(\mathbb{OP}^2 \times N^m)$.

In applications where the space with periodic cohomology is a closed, orientable manifold, Poincar\'e duality imposes additional multiplicative structure. This was used in the proof of \cite[Lemma 3.3]{Kennard17}, which verifies the truncated conjecture for $p \geq 3$ and $k = 2p$ in the odd-dimensional Poincar\'e duality manifold case. Similarly, in the further special case where periodicity arises from $n$-connected inclusions $N^n \to M^{n+k}$ of Poincar\'e duality manifolds, perhaps other geometric input could be applied to prove the conjecture.

\bigskip
\section{Structure of potential counterexamples}\label{sec:PotentialCounterexamples}
\bigskip

Theorem \ref{thm:2} implies that a counterexample $X$ to Conjecture \ref{con:2} with $k = 16$ has non-trivial cohomology in degrees $i \equiv 8 \bmod 16$. Moreover, the discussion following the proof of Lemma \ref{lem:4} implies that such a counterexample has additional non-trivial cohomology in some degrees $i \not\equiv 0 \bmod 8$. In this section, we further analyze the structure of a potential counterexample with $k = 16$ and the property that $H^i(X; \Z_2) = 0$ for all $i \not\equiv 0 \bmod 4$. In fact, our analysis works more generally under the assumption that $\Sq^i = 0$ for all $i \not\equiv 0 \bmod 4$.

\begin{lemma}\label{lem:16}
Let $X$ be a space for which $H^*(X; \Z_2)$ is $16$-periodic but not $8$-periodic. Assume in addition that $\Sq^1 = 0$ and $\Sq^2 = 0$ on $H^*(X; \Z_2)$. Let $x \in H^{16}(X; \Z_2)$ denote the element inducing periodicity, and let $y_i \in H^i(X; \Z_2)$ denote the elements satisfying the equation $\Sq^i x = x y_i$ for all $i \in \{4, 8, 12\}$. All of the following hold:
	\begin{enumerate}
	\item\label{id} The map $x^{-1} \Sq^{16}$ equals the identity on $H^{16+i}(X; \Z_2)$ for all $i \in \{0,4,8,12\}$, where $x^{-1}$ denotes the inverse of the isomorphism given by multiplication by $x$.
	\item\label{big_products} Products of the form $u_8 v_{12}$, $u_{12} v_{12}$, and $u_8 \Sq^8(v_{12})$ vanish, where subscripts denote degrees.
	\item\label{small_products} Products of the form $y_i y_j$ vanish for all $i$ and $j$, including $i = j$.
	\item\label{y12} $\Sq^4 y_8 = y_{12}$ and $\Sq^8 y_{12} = x y_4$.
	\end{enumerate}
\end{lemma}

\begin{proof}
Part \eqref{id} is immediate for $i = 0$ and is straightforward for $i = 4$ since $x^2$ does not non-trivially factor. The proof for $i \in \{8, 12\}$ requires results on products, so we prove these next.

To prove the first claim of Part \eqref{big_products}, suppose that $u_8 v_{12} = x w_4 \neq 0$. Applying $\Sq^{16}$ to both sides and applying \eqref{id}, we have $\Sq^4(u_8) \Sq^{12}(v_{12}) = x^2 w_4 \neq 0$. Pulling out the $\Sq^4$ using the Cartan formula, we have that $u_8 v_{12}^2 = x^2$ since $u_8 \Sq^4(v_{12}^2) = u_8 \Sq^2(v_{12})^2 = 0$. This is a contradiction since $x^2$ does not non-trivially factor.

We now prove Part \eqref{id} for $i = 8$. The Cartan formula implies $\Sq^{16}(x t_8) = x^2 t_8 + x y_{12} \Sq^4(t_8)$ since $t_8^2 = 0$. The claim follows because $y_{12} \Sq^4(t_8) = \Sq^4(y_{12} t_8) - \Sq^4(y_{12}) t_8$, which vanishes by the first claim of \eqref{big_products} together with the fact that $x$ is not in the image of a non-trivial Steenrod operation.

The last claim of Part \eqref{id} follows similarly, so we omit the proof.

Next we prove the second claim in \eqref{big_products}, suppose $u_{12} v_{12} = x w_8 \neq 0$. By Part \eqref{id}, applying $\Sq^{16}$ yields \[x^2 w_8 = \Sq^8(u_{12}) \Sq^8(v_{12}) = \Sq^8\of{u_{12} \Sq^8(v_{12})} - \Sq^4(u_{12})\Sq^{4}\Sq^8 v_{12} - u_{12} \Sq^8\Sq^8 v_{12}.\] The first two terms on the right-hand side vanish by the minimality properties of $x$, and the last also vanishes by using the Adem relation for $\Sq^8 \Sq^8$. 

The last claim of Part \eqref{big_products} follows by the Cartan formula, which implies $u_8 \Sq^8(v_{12}) = \Sq^8(u_8 v_{12}) - \Sq^4(u_8) \Sq^4(v_{12}) - u_8^2 v_{12}$, which vanishes by the first part of \eqref{big_products} and the properties of $x$ following from minimality.

To prove Part \eqref{small_products}, first consider the $x y_4^2 = y_4 \Sq^4 x$. Applying $\Sq^{16}$ to both sides and applying \eqref{id}, we obtain
	\[x^2 y_4^2 =  \Sq^{16}(y_4 \Sq^4 x) = y_4 \Sq^{16}\Sq^4 x = y_4 \Sq^8 \Sq^{12} x.\]
Applying the Cartan formula to the action of $\Sq^8$, we have
	\[y_4 \Sq^8 \Sq^{12} x = \Sq^8\of{y_4 \Sq^{12} x} - y_4^2 \Sq^4 \Sq^{12} x = 0.\]
Hence $x y_4^2 = 0$, which implies $y_4^2$ vanishes by periodicity. 

To finish the proof of Part \eqref{small_products}, we only need to prove $y_4 y_8 = 0$ since the other products vanish by minimality of $x$. For this we follow a similar strategy as before.  Consider the term $x y_4 y_8 = y_8 \Sq^4 x$. Applying $\Sq^{16}$ yields
	\[x^2 y_4 y_8 = y_8 \Sq^{16} \Sq^4 x + \Sq^4 y_8 \Sq^{12} \Sq^4 x.\]
The first term on the right-hand side vanishes as in the previous argument, using in addition the fact that $y_8 \Sq^{12} x = 0$ by Part \eqref{big_products}. The second term on the right-hand side also vanishes because $x^2$ does not decompose non-trivially as the image of a Steenrod operation. Hence the right-hand side is zero in contradiction to Part \eqref{id}.

To prove Part \eqref{y12}, we first apply Part \eqref{small_products} to calculate 
	\[x y_{12} = \Sq^{12}(x) = \Sq^4 \Sq^8 x = \Sq^4(x y_8) = x \Sq^4 y_8,\]
which proves $y_{12} = \Sq^4 y_8$. Similarly, we can apply Part \eqref{big_products} and the minimality properties of $x$ to calculate
	\[x \Sq^8 y_{12} = \Sq^8(x y_{12}) = \Sq^8 \Sq^{12} x = \Sq^{20} x + \Sq^{16} \Sq^4 x = \Sq^{16}(x y_4).\]
Multiplying by $x^{-1}$ and applying Part \eqref{id} yields the claim.
\end{proof}

We can push this analysis further with the use of secondary cohomology operations. 

\begin{lemma}
If $X$ is a counterexample to Conjecture \ref{con:2} with $k = 16$, and if $\Sq^1$ and $\Sq^2$ vanish on $H^*(X; \Z_2)$, then $\Sq^8 \Sq^4 y_8 = x y_4 \neq 0$ where $\Sq^i x = x y_i$ for $i \in \{4, 8\}$.
\end{lemma}

\begin{proof}
We assume that $x \in H^{16}(X; \Z_2)$ induces periodicity and that $16$ is the minimum period. 
We assume also that $\Sq^1 = 0$ and $\Sq^2 = 0$ everywhere. 
In particular, Lemma \ref{lem:16} applies. 
If $\Sq^4(x) \neq 0$, then Lemma \ref{lem:16} implies that $\Sq^8 \Sq^4 y_8 = \Sq^8 y_{12} = x y_4 \neq 0$, as claimed. 
To finish the proof, we assume $\Sq^4(x) = 0$ and seek a contradiction.

Lemma \ref{lem:16} only used primary operations in the proof. Here we use the following secondary decompositions:
	\begin{enumerate}
	\item[(a)] If $u \in \ker(\Sq^1, \Sq^2, \Sq^4, \Sq^8)$, then 
		$\Sq^{16} u = \Sq^8(v) + \Sq^4 \Sq^8(v') + \Sq^4\Sq^2\Sq^1(v'')$ 
	modulo $\im(\Sq^1, \Sq^2)$ for some $v$, $v'$, and $v''$.
	\item[(b)] If $u \in \ker(\Sq^1, \Sq^2, \Sq^8) \cap \ker(\Sq^2\Sq^4, \Sq^8\Sq^4)$, then $\Sq^{16} u = \Sq^4(w) + \Sq^8(w')$ modulo $\im(\Sq^1, \Sq^2)$ for some $w$ and $w'$.
	\end{enumerate}
The first decomposition follows from \cite[Theorem 6.2.1]{Harper02}. For the second one, we refer to \cite[Section 3]{LinWilliams91}. We note first that the $\alpha_{13}$ in the notation of \cite{LinWilliams91} should be $\alpha_{13} = \Sq^3 + \Sq^2 \Sq^1$. We also remark that Lin and Williams state the secondary decomposition only for integral cohomology elements $u$. However, there is a standard procedure for padding the matrix $B$ and vector $c = (\Sq^2, \Sq^6, \Sq^8, \Sq^8\Sq^4)^T$ in the equation at the top of \cite[page 180]{LinWilliams91} and adding an entry to the vector \[a = (\alpha_3, \alpha_8, \alpha_9, \alpha_{12}, \alpha_{13}, \alpha_{15})\] to extend the decomposition to all (not necessarily integral) elements as long as, in addition, they lie in the kernel of $\Sq^1$. This is sufficient to obtain (b). We omit the details.

By Parts \eqref{big_products} and \eqref{y12} of Lemma \ref{lem:16}, we have $x y_8 \in \ker(\Sq^8\Sq^4)$. It is then easy to check that secondary decomposition (b) applies to $x y_8$. By Part \eqref{id}, we have
	\[x^2 y_8 = \Sq^{16}(x y_8) = \Sq^4(x z_4) + \Sq^8(x^2 z_0),\]
where $z_4 \in H^4(X; \Z_2)$ and $z_0 \in H^0(X; \Z_2)$. Since $y_4 = 0$, we conclude 
	\[y_8 = z_4^2.\]
We claim that $z_4^3 = 0$. Indeed, $x z_4^3 \in \ker(\Sq^4, \Sq^8)$, so conditional relation (a) applies. By Part \eqref{id}, we have
	\[x^2 z_4^3 = \Sq^{16}(x z_4^3) = \Sq^8(x^2 a_4) + \Sq^4\Sq^8(x^2 a_0) + \Sq^4\Sq^2\Sq^1(x^2 a_5)\]
for some $a_4 \in H^4(X; \Z_2)$, $a_0 \in H^0(X; \Z_2)$, and $a_5 \in H^5(X; \Z_2)$. All terms on the right-hand side vanish by the Cartan formula and the assumptions that $y_4 = 0$ and $\Sq^1 = 0$, so we have that
	\[z_4^3 = 0,\]
as claimed. It now follows that $x z_4 \in \ker(\Sq^8)$ since $y_8 = z_4^2$ and $y_4 = 0$. But we also have $x z_4 \in \ker(\Sq^8 \Sq^4)$ since this lands in the group generated by $x^2$, which does not decompose by the minimality properties of $x$. Hence conditional relation (b) applies to $x z_4$ and we have a relation of the form
	\[x^2 z_4 = \Sq^{16}(x z_4) = \Sq^4(x^2 b_0) + \Sq^8(x b_{12}).\]
As before the first term on the right-hand side is zero, and the other equals $x \Sq^8(z_{12})$. We now have
	\[x y_8 = x z_4^2 = \Sq^8(z_{12}) z_4 = \Sq^8(z_{12} z_4) + \Sq^4(z_{12}) z_4^2.\]
Both terms on the right-hand side vanish by the minimality properties of $x$, so we
	\[y_8 = 0.\]
But now the proof of the main theorem applies. In short, we can now prove that conditional relation (a) applies to $x$, and the minimality properties of $x$ imply that 
	\[y_{16} = 0,\]
where for suggestive reasons we have denoted $x$ by $y_{16}$ since they both satisfy the equation $\Sq^{16}(x) = x y_{16}$. This is a contradiction, so the proof is complete.
\end{proof}

To motivate further investigation, we propose the following as a minimal potential counterexample to Conjecture \ref{con:2} when $k = 16$.

\begin{question}
Does there exist a space X with $\Z_2$-cohomology ring isomorphic to $\Z_2[y_4, y_8, y_{12}, x_{16}]/(y_4^2, y_4 y_8, y_4 y_{12}, y_8^2, y_8 y_{12}, y_{12}^2)$ and the properties that 
$\Sq^i x_{16} = x_{16} y_i$ for all $i \in \{4, 8, 12\}$, $\Sq^4 y_8 = y_{12}$, and $\Sq^8 y_{12} = x y_4$. 
\end{question}


\begin{thebibliography}{KWW1}

\bibitem[Ada60]{Adams60}
J.F. Adams.
\newblock {On the non-existence of elements of Hopf invariant one}.
\newblock {\em Ann. of Math.}, 72(1):20--104, 1960.

\bibitem[Ade52]{Adem52}
J.~Adem.
\newblock {The iteration of the Steenrod squares in algebraic topology}.
\newblock {\em Proc. Nat. Acad. Sci. USA}, 38:720--726, 1952.

\bibitem[Gon77]{Goncalves77}
D.L. Goncalves.
\newblock {\em {Mod $2$ homotopy-associative $H$-spaces}}.
\newblock ProQuest LLC, Ann Arbor, MI, 1977.
\newblock Thesis (Ph.D.)--University of Rochester.

\bibitem[Har02]{Harper02}
J.~Harper.
\newblock {\em {Secondary cohomology operations}}.
\newblock American Mathematical Society, 2002.

\bibitem[Ken13]{Kennard13}
L.~Kennard.
\newblock {On the Hopf conjecture with symmetry}.
\newblock {\em Geom. Topol.}, 17:563--593, 2013.

\bibitem[Ken17]{Kennard17}
L.~Kennard.
\newblock Fundamental {G}roups of {M}anifolds with {P}ositive {S}ectional
  {C}urvature and {T}orus {S}ymmetry.
\newblock {\em J. Geom. Anal.}, 27(4):2894--2925, 2017.

\bibitem[KWW1]{KWW1}
L.~Kennard, M.~Wiemeler, and B.~Wilking.
\newblock Splitting of torus representations and applications in the grove
  symmetry program.
\newblock {\em preprint}, arXiv:2106.14723, {\noopsort{1}}.

\bibitem[Liu62]{Liulevicius62}
A.~Liulevicius.
\newblock {The factorization of cyclic reduced powers by secondary cohomology
  operations}.
\newblock {\em Mem. Amer. Math. Soc.}, 42, 1962.

\bibitem[LW90]{LinWilliams90}
J.P. Lin and F.~Williams.
\newblock Two torsion and homotopy associative {$H$}-spaces.
\newblock {\em J. Math. Kyoto Univ.}, 30(3):523--541, 1990.

\bibitem[LW91]{LinWilliams91}
J.P. Lin and F.~Williams.
\newblock The type of a torsion free finite loop space.
\newblock {\em Topology Appl.}, 42(2):175--186, 1991.

\bibitem[Nie]{Nienhaus-pre}
J.~Nienhaus.
\newblock {An improved four-periodicity theorem and a conjecture of Hopf with
  symmetry}.

\bibitem[SY61]{ShimadaYamanoshita61}
N.~Shimada and T.~Yamanoshita.
\newblock {On triviality of the mod $p$ Hopf invariant}.
\newblock {\em Jpn. J. Math.}, 31:1--25, 1961.

\bibitem[Wil03]{Wilking03}
B.~Wilking.
\newblock {Torus actions on manifolds of positive sectional curvature}.
\newblock {\em Acta Math.}, 191(2):259--297, 2003.

\end{thebibliography}

\end{document}